\documentclass[12pt,oneside]{amsart}
\usepackage{amssymb,verbatim}

\numberwithin{equation}{section}
\newtheorem{theorem}{Theorem}[section]
\newtheorem{lemma}[theorem]{Lemma}

\newtheorem{corollary}[theorem]{Corollary}
\newtheorem{conjecture}[theorem]{Conjecture}

\theoremstyle{definition}
\newtheorem{definition}[theorem]{Definition}
\newtheorem{remark}[theorem]{Remark}
\newtheorem{example}[theorem]{Example}

\DeclareMathOperator{\Ass}{Ass}

\newcommand{\K}{\mathcal{K}}

\newcommand{\N}{\mathbb{N}}

\newcommand{\mf}{\mathfrak{m}}

\begin{document}


\title[Critical graphs and associated primes]
{A conjecture on critical graphs and 
connections to the persistence of associated primes}

\author{Christopher A. Francisco}
\address{Department of Mathematics, Oklahoma State University,
401 Mathematical Sciences, Stillwater, OK 74078}
\email{chris@math.okstate.edu}
\urladdr{http://www.math.okstate.edu/$\sim$chris}

\author{Huy T\`ai H\`a}
\address{Tulane University \\ Department of Mathematics \\
6823 St. Charles Ave. \\ New Orleans, LA 70118, USA}
\email{tai@math.tulane.edu}
\urladdr{http://www.math.tulane.edu/$\sim$tai/}

\author{Adam Van Tuyl}
\address{Department of Mathematical Sciences \\
Lakehead University \\
Thunder Bay, ON P7B 5E1, Canada}
\email{avantuyl@lakeheadu.ca}
\urladdr{http://flash.lakeheadu.ca/$\sim$avantuyl/}

\thanks{Version: April 19, 2010}

\begin{abstract}
We introduce a conjecture about constructing critically $(s+1)$-chromatic graphs
from critically $s$-chromatic graphs.  We then show how this conjecture implies that any unmixed height two square-free monomial ideal $I$ in a polynomial ring $R$, i.e., the cover ideal of a finite simple graph, has the persistence property, that is, $\Ass(R/I^s) \subseteq \Ass(R/I^{s+1})$
for all $s \geq 1$. 
To support our conjecture, we prove that the statement is true if we also assume that 
$\chi_f(G)$,  fractional chromatic number of
the graph $G$, satisfies $\chi(G) -1 < \chi_f(G) \leq \chi(G)$.  We give an
algebraic proof of this result.
\end{abstract}

\maketitle


\section{Introduction}
Let $G = (V_G,E_G)$ be a finite simple graph
with vertex set $V_G = \{x_1,\ldots,x_n\}$ and edge set $E_G$.  We say that $G$ has a {\bf $s$-coloring}
if there exists a partition $V_G = C_1 \cup \cdots \cup C_s$ such that for every $e \in E_G$,
$e \not\subseteq C_i$ for $i=1,\ldots,s$.  The minimal integer $s$ such that $G$ has
a $s$-coloring is called the {\bf chromatic number} of $G$, and is denoted $\chi(G)$.  The coloring
of graphs is one of the main branches in the field of graph theory and has
many applications to other fields. In this short note,
we propose a conjecture about the coloring of critical graphs that arose out of our
study of an algebraic question about the associated primes of powers of square-free monomial 
ideals. 

To state our conjecture, we recall the following definitions and construction.  A graph
$G$ is said to be {\bf critically $s$-chromatic} if $\chi(G) = s$ but $\chi(G \setminus x) = s-1$
for every $x \in V_G$, where $G \setminus x$ denotes the graph obtained from $G$ by removing the vertex $x$ and all edges incident to $x$. A graph that is critically $s$-chromatic for some $s$ is called {\bf critical}.   For any vertex $x_i \in V_G$,
the {\bf expansion} of $G$ at the vertex $x_i$ is the graph $G' = G[\{x_i\}]$
whose vertex set is given by $V_{G'} = (V_G \setminus \{x_i\}) \cup \{x_{i,1},x_{i,2}\}$ and
whose edge set has form
\[E_{G'} = \{\{u,v\} \in E_G ~|~ u \neq x_i ~~\mbox{and}~~ v \neq x_i\} \cup
\{\{u,x_{i,1}\},\{u,x_{i,2}\} ~|~ \{u,x_i\} \in E_G\} \cup \{\{x_{i,1},x_{i,2}\}\}.\]
Equivalently, $G[\{x_i\}]$ is formed by replacing the vertex $x_i$ with the clique
$\K_2$ on the vertex set $\{x_{i,1},x_{i,2}\}$.  For any $W \subseteq V_G$,
the {\bf expansion} of $G$ at $W$, denoted $G[W]$, 
is formed by successively expanding all the vertices of $W$ (in any order).  We propose the following conjecture:

\begin{conjecture}\label{conjecture}
Let $s$ be a positive integer, and let $G$ be a finite simple graph that is critically
$s$-chromatic.  Then there exists a subset $W \subseteq V_G$ such
that $G[W]$ is a critically $(s+1)$-chromatic graph.
\end{conjecture}

If Conjecture \ref{conjecture} holds, then given any critically $s$-chromatic graph,
one can then construct a critically $(s+d)$-chromatic graph for any integer $d \geq 1$
by repeatedly applying the conjecture.  Note that Conjecture \ref{conjecture} is true
if we also assume that $G$ is a perfect graph since the only perfect
critically $s$-chromatic graph is $\K_s$, the clique of size $s$.  If we expand $\K_s$
at any vertex, we obtain the critically $(s+1)$-chromatic graph $\K_{s+1}$.

Conjecture \ref{conjecture} arose out of our investigations of the associated primes
of powers of the cover ideal of a graph \cite{FHVT2}. Let $k$ be a field, and let $R = k[x_1, \dots, x_n]$ be a polynomial ring over $k$. 
For any ideal $I \subseteq R$, a prime ideal $P$ is an {\bf associated prime} of $R/I$ if there
exists an element $T \in R$ such that $I:(T) = \{F \in R ~|~ FT \in I\} = P$.
The set of associated primes of $R/I$ is denoted by $\Ass(R/I)$.  
Consider the sets $\Ass(R/I^s)$ as $s$ varies. Then 
Brodmann (see \cite{Brodmann}) proved that there exists an integer $a$ such that
\[\bigcup_{s=1}^\infty \Ass(R/I^s) = \bigcup_{s=1}^a \Ass(R/I^s).\]
However, it is not true in general that 
$\Ass(R/I^s) \subseteq \Ass(R/I^{s+1})$ for all $s \geq 1$ (see
the paper of Herzog and Hibi \cite{HH} for examples
involving monomial ideals).
We say that $I$ has the {\bf persistence property} if $\Ass(R/I^s) \subseteq
\Ass(R/I^{s+1})$ for all $s \geq 1$.

Conjecture \ref{conjecture} implies the persistence property for 
the following class of ideals:

\begin{theorem} \label{maintheorem}
Suppose that $I$ is any unmixed square-free monomial ideal of height
$2$.  If Conjecture \ref{conjecture} holds for $(s+1)$, then 
$$\Ass(R/I^s) \subseteq \Ass(R/I^{s+1}).$$
In particular, if Conjecture \ref{conjecture} holds for all $s$ then $I$ has the persistence property.
\end{theorem}
\noindent
Any unmixed square-free monomial ideal of height two is the Alexander dual
of an edge ideal of a graph (for an introduction, see Chapter 6 of \cite{V}).
We prove Theorem \ref{maintheorem} in Section 2 by using our
results  in \cite{FHVT2} which link the irredundant
irreducible decomposition of a power of a square-free monomial ideal to critical subgraphs.

In Section 3, we prove Conjecture \ref{conjecture} if $\chi_f(G)$, the
fractional chromatic number of $G$, is ``close enough'' to $\chi(G)$.

\begin{theorem}\label{maintheorem2}
Conjecture \ref{conjecture} holds if we also assume 
$\chi(G)-1 < \chi_f(G) \leq \chi(G)$.  
\end{theorem}

\noindent
We give an algebraic proof of Theorem \ref{maintheorem2}.  As a corollary,
all odd holes and odd antiholes, both examples 
of critical graphs, satisfy Conjecture \ref{conjecture}.

We view this paper as part of the ongoing dialogue between 
graph coloring problems and commutative algebra.  As examples
of this discussion, we point the reader to the following papers.
In \cite{SS}, Sturmfels and Sullivant identify the generators of secant ideals of the 
edge ideal of a graph $G$ in terms of the colorability of induced subgraphs of $G$
in order to give an algebraic interpretation of the Strong Perfect Graph Theorem. 
A previous paper of the authors \cite{FHVT2}, which extends results of \cite{SS}, examines
how the associated primes of the powers of the cover ideal encode coloring
information of the graph.  Moreover, this work inspired the conjecture 
of this paper.
In a different direction, Steingr\'imsson constructed a square-free monomial ideal $K_G$ in a 
polynomial ring $R$ where the monomials of $K_G$ are in one-to-correspondence with the 
colorings of the graph $G$ \cite{S}. Additionally, there is a nice CoCoA tutorial that investigates a Gr\"obner basis method for determining whether a graph can be colored with three colors \cite[Tutorial 26]{KR}. 
Connections between the chromatic number and ideals that are not necessarily monomial can be
found in \cite{Brown,DHMO}.   As well, other variations on graph coloring can be
tackled using commutative algebra, as in the paper of Miller \cite{Mi} where orthogonal colorings
of planar graphs and connections to the minimal free resolutions of monomial
ideals of $k[x,y,z]$ are studied.

\begin{remark} \label{r.generalize}
As a final comment, one can also formulate a hypergraph
version of Conjecture \ref{conjecture}, and by adapting the proof of Theorem
\ref{maintheorem}, one can remove the unmixed and height hypotheses 
from Theorem \ref{maintheorem}. Thus a proof of the hypergraph version of Conjecture~\ref{conjecture} 
would prove the persistence property for \emph{any} square-free monomial ideal. We
have decided to emphasize the 
graph theory question, at the expense of the more general discussion, in the hope
of attracting graph theorists to this problem.  As well, we have less evidence
that the hypergraph analog of Conjecture \ref{conjecture} is true.
\end{remark}

\noindent {\bf Acknowledgments.} CoCoA \cite{C} and Macaulay 2 \cite{M2}
helped inspire our results.  We would like to thank
 Bjarne Toft and  Anders Sune Pedersen 
for answering our questions about critical graphs, and
the referee for their corrections and suggestions.
The first author is partially
supported by a NSA Young Investigator's Grant and an Oklahoma State University Dean's
Incentive Grant. The second author is partially supported by the Board of Regents grant
LEQSF(2007-10)-RD-A-30 and Tulane's Research Enhancement Fund. The third author acknowledges the
support provided by NSERC.


\section{Persistence of cover ideals}

The goal of this section is to prove Theorem \ref{maintheorem}.
We begin by encoding a finite simple graph into a monomial ideal. 

\begin{definition}
Let $G$ be a finite simple graph on the vertex set $V_G = \{x_1,\ldots,x_n\}$.  The
{\bf cover ideal} of $G$ is the monomial ideal
\[J = J(G) = \bigcap_{\{x_i,x_j\} \in E_G} (x_i,x_j) \subseteq R = k[x_1,\ldots,x_n].\]
\end{definition}
From the definition, it follows that there is a one-to-one correspondence
between cover ideals of finite simple graphs (with at least one edge and no isolated vertices) and unmixed square-free monomial ideals of height two.

Recall that a subset $W \subseteq V_G$ is called a {\bf vertex cover} of $G$ if 
$e \cap W \neq \emptyset $ for every $e \in E_G$.  
A vertex cover is a {\bf minimal vertex cover}
if no proper subset is a vertex cover.  An subset $W$ is an {\bf independent
set} if $V_G \setminus W$ is a vertex cover, and is a {\bf maximal
independent set} if $V_G \setminus W$ is a minimal vertex cover.
 The name cover ideal makes sense
in light of the following lemma.

\begin{lemma}  Let $G$ be a finite simple graph.  Then
\[J(G) = (x_{i_1}\cdots x_{i_r} ~|~ W = \{x_{i_1},\ldots,x_{i_r}\} ~~\mbox{is a minimal vertex cover of $G$}).\]
\end{lemma}

Irreducible monomial ideals will play a large role in the following discussion:

\begin{definition}
A monomial ideal in $R = k[x_1,\ldots,x_n]$ of the form $\mf^{{\bf b}} = (x_i^{b_i} ~|~ b_i \geq 1)$ with ${\bf b} = (b_1, \dots, b_n) \in \N^n$ is called an {\bf irreducible ideal}.  An {\bf irreducible decomposition} of a monomial ideal $I$ is an expression of the form
\[ I = \mf^{{\bf b}_1} \cap \cdots \cap \mf^{{\bf b}_r} \]
for some vectors ${\bf b}_1,\ldots,{\bf b}_r \in \N^n$.  The decomposition
is {\bf irredundant} if none of the $\mf^{{\bf b}_i}$ can be omitted.
\end{definition}

Every monomial ideal has a unique
irredundant irreducible decomposition \cite[Theorem 5.27]{MS}.  
In our previous paper \cite{FHVT2}, we explored what information
about the graph $G$ is encoded into the irredundant irreducible decompositions of the ideals
$J(G)^s$ as $s$ varies, that is, the powers of the cover ideals.  To state our results, we need the following 
terminology.

\begin{definition} \label{expansion}
Let $G = (V_G,E_G)$ be a graph with vertices $V_G = \{x_1, \dots, x_n\}$.
For each $s$, we define the {\bf $s$-th expansion} of $G$ to be the graph
obtained by replacing each vertex $x_i \in V_G$ by a collection $\{x_{ij} ~|~ j = 1, \dots, s\}$,
and replacing $E_G$ by the edge set that consists of edges $\{x_{il_1},x_{jl_2}\}$
whenever $\{x_i,x_j\} \in E_G$ and edges
$\{x_{il}, x_{ik}\}$ for $l \not= k$. We denote this graph by $G^s$.
We call the new
variables $x_{ij}$ the {\bf shadows} of $x_i$. 
\end{definition}

Observe that in the $s$-th expansion of $G$, we are replacing each vertex of $G$
with a clique of size $s$.    The connection between the irredundant irreducible decomposition
of $J(G)^s$ and the graph $G^s$ is summarized in the following theorem.

\begin{theorem}[{\cite[Corollary 4.14]{FHVT2}}] \label{criticalcorrespondence}
Let $G$ be a finite simple graph with cover ideal $J(G)$.  The following
are equivalent:
\begin{enumerate}
\item[(i)] $(x_{i_1}^{a_{i_1}},\ldots,x_{i_r}^{a_{i_r}})$ appears in the irredundant irreducible
decomposition of $J(G)^s$
\item[(ii)] The induced subgraph of $G^s$ on the vertex set 
\[Y = \{x_{i_1,1},\ldots,x_{i_1,s-a_{i_1}+1},\ldots,x_{i_r,1},\ldots,x_{i_r,s-a_{i_r}+1}\},\]
is a critically $(s+1)$-chromatic graph.
\end{enumerate}
\end{theorem}

For any monomial ideal $I$,  $(x_{i_1},\ldots,x_{i_r}) \in \Ass(R/I)$
if and only if there exists an irreducible ideal of the form
$(x_{i_1}^{a_{i_1}},\ldots,x_{i_r}^{a_{i_r}})$ in the irredundant irreducible decomposition
of $I$.  We come to the main result of this section, which shows that if Conjecture~\ref{conjecture} is true, unmixed height
two square-free monomial ideals have the persistence property.  

\begin{theorem} \label{persistence-graphs}
Let $G$ be a finite simple graph with cover ideal $J =J(G)$. Let $s \ge 1$ and assume that Conjecture \ref{conjecture} holds for $(s+1)$. Then
\[\Ass(R/J^s) \subseteq \Ass(R/J^{s+1}).\]
In particular, if Conjecture \ref{conjecture} holds for all $s$ then $J$ has the persistence property.
\end{theorem}

\begin{proof} Let 
$P = (x_{i_1},\ldots,x_{i_r})\in \Ass(R/J^s)$.  To simplify our notation, we can
assume, after relabeling, that $P = (x_1,\ldots,x_r)$.  Thus,
there exist positive integers $a_1,\ldots,a_r$ such that 
$(x_1^{a_1},\ldots,x_r^{a_r})$ appears in the irredundant irreducible decomposition
of $J^s$.   By Theorem \ref{criticalcorrespondence}, this means that the 
induced graph of $G^s$ on the vertex set 
\[Y = \{x_{1,1},\ldots,x_{1,s-a_{1}+1},\ldots,x_{r,1},\ldots,x_{r,s-a_{r}+1}\}\]
is a critically $(s+1)$-chromatic graph.  Let $H = (G^s)_Y$ denote this
induced subgraph.

Since we are assuming that Conjecture \ref{conjecture} holds for $(s+1)$, we can
find a subset $W \subseteq Y$ such that the expansion of $H$ at $W$,
that is, the graph $H[W]$, is a critically $(s+2)$-chromatic graph.  
For each $i=1,\ldots,r$, let $b_i$ denote the number of shadows of $x_i$
that are contained in $W$.  

We claim that $0 \leq b_i \leq a_i$ for each $i$.  Indeed, suppose that
$W$ contains $b_i \geq a_i+1$ shadows of $x_i$.  The induced graph
of $H$ on $\{x_{i,1},\ldots,x_{i,s-a_i+1}\}$ is a clique of size $s-a_i+1$.
If we expand $H$ at the $b_i$ shadows of $x_i$ in $W$, we end up with
a clique of size $b_i + s-a_i+1 \geq a_i+1 + s-a_i+1 = s+2$, i.e., 
$H[W]$ will contain a clique of size at least $s+2$ as an induced subgraph.
On the other hand, because $H$ is a critically $(s+1)$-chromatic graph,
and because the induced graph on $\{x_{i,1},\ldots,x_{i,s-a_i+1}\}$ has size at
most $s-a_i+1 \leq s < s+1$, the graph $H$ has at least one other edge.  But
then this edge (or a shadow of this edge) also belongs to $H[W]$ and does not belong to the 
clique of size of $s+2$ in $H[W]$.  We now have a contradiction, since
$H[W]$ is a critically $(s+2)$-chromatic graph that contains a $\K_{s+2}$
as a proper induced subgraph.

Now consider the induced subgraph of $G^{s+1}$ on the vertex set
\[Y' = \{x_{1,1},\ldots,x_{1,s-a_{1}+b_1+1},\ldots,x_{r,1},\ldots,x_{r,s-a_{r}+b_r+1}\}.\]
By comparing the constructions of $(G^{s+1})_{Y'}$ and $H[W]$, one can verify
that these two graphs are isomorphic.  Thus $(G^{s+1})_{Y'}$ is a critically $(s+2)$-subgraph
of $G^{s+1}$.
Theorem \ref{criticalcorrespondence} then implies that the irreducible ideal 
\[(x_1^{a_1-b_1+1},\ldots,x_r^{a_r-b_r+1})\]
appears in the irredundant irreducible decomposition of $J(G)^{s+1}$.  Since
$0 \leq b_i \leq a_i$ for each $i$, we have $a_i-b_i + 1 \geq 1$ for each $i$.
Hence $(x_1,\ldots,x_r) \in \Ass(R/J^{s+1})$, as desired.
\end{proof}

\begin{remark}
In \cite{FHVT2}, we proved that $J = J(G)$ has the persistence property 
if $G$ is perfect.  Our proof relies heavily on the fact that
Conjecture \ref{conjecture} is true for perfect critical graphs.
\end{remark}


\section{Special cases of the conjecture}

In this section, we prove that Conjecture \ref{conjecture} holds for all $s$ 
if the fractional chromatic number of $G$ is ``close enough'' to $\chi(G)$.
We begin with a lemma that shows that we can prove Conjecture \ref{conjecture} 
if there exists a maximal independent set $W \subseteq V_G$
such that $\chi(G[W]) = \chi(G) +1$.

\begin{lemma}  \label{independentlemma}
Suppose that $G$ is a critically $s$-chromatic graph and $W$ is
a maximal independent set such that $\chi(G[W]) = s+1$.  Then
$G[W]$ is critically $(s+1)$-chromatic.
\end{lemma}

\begin{proof}
It suffices to prove that $G[W]$ is \emph{critically} $(s+1)$-chromatic since
we are given $\chi(G[W]) = s+1$. After relabeling the vertices,
we can assume that $W = \{x_1,\ldots,x_r\}$ and $V_G \setminus W = \{x_{r+1},\ldots,x_n\}$.
Thus, the vertices of $G[W]$ are 
\[V_{G[W]} = \{x_{11},x_{12},\ldots,x_{r1},x_{r2},x_{r+1},\ldots,x_n\}.\]

Fix an $x_j \in \{x_{r+1},\ldots,x_n\}$, and set $H = G \setminus x_j$.
The $H[W] = G[W]\setminus x_j$.  We can color $G \setminus x_j$ with $s-1$ colors,
and in particular, we use the $s-1$ colors to color the induced subgraph of
$H[W]$ on $\{x_{11},\ldots,x_{r1},x_{r+1},\ldots,x_n\} \setminus \{x_j\}$. We then
color the vertices $\{x_{12},\ldots,x_{r2}\}$ with the $s$-th color.  Since
$W$ is an independent set, we can color all of these vertices the same color.
But this means that $H[W] = G[W] \setminus x_j$ is $s$-colorable.

So, now fix an $x_{j1} \in \{x_{11},x_{21},\ldots,x_{r1}\}$.  We will show
that we can color $G[W] \setminus x_{j1}$ with $s$ colors.  If we remove
$x_j$ from $G$, we can color $G \setminus x_j$ with $s-1$ colors.  This
gives us a $(s-1)$-coloring of the vertices of $\{x_{11},\ldots,x_{r1},x_{r+1},\ldots,x_n\} \setminus
\{x_{j1}\}$.  The independent set $\{x_{12},\ldots,x_{r2}\}$ in $G[W] \setminus x_{j1}$
can now be colored with the $s$-th color.  A similar argument
works if $x_{j2} \in \{x_{12},x_{22},\ldots,x_{r2}\}$ is removed.
\end{proof}

\begin{remark} Suppose $G$ is critically $s$-chromatic, and let $W$ be a non-maximal independent set in $G$; that is, there exists a vertex $x$ so that $W \cup \{x\}$ is an independent set. It can be seen that $\chi(G \backslash x) = s-1$, so by assigning the $s$-th color to $x$ and shadows of the vertices in $W$, we have $\chi(G[W]) = s$. This says that for $G[W]$ to be critically $(s+1)$-chromatic, $W$ cannot be a non-maximal independent set. On the other hand, the proof of Lemma \ref{independentlemma} suggests
that the easiest way to prove Conjecture \ref{conjecture} is to show the 
existence of a maximal independent set $W \subseteq V_G$ with the property
that $\chi(G[W]) = \chi(G) +1$.  Our initial hope was that any maximal
independent set, or at the very least, any maximum independent
set (a maximal independent set with largest possible cardinality), would have the desired property.  The following two examples, due to Bjarne Toft and  Anders Sune Pedersen, show that this is not true.
\end{remark}

\begin{example}
Let $G = C_9$, the cycle on $V_G = \{x_1,\ldots,x_9\}$.  Consider the maximal independent
set $W = \{x_2,x_5,x_8\}$.  When we construct the graph $G[W]$, the resulting
graph has $\chi(G[W]) = \chi(G) = 3$.
\end{example}

\begin{example}
As in the previous example, we begin with the cycle $C_9$ on the vertex
set $V_G = \{x_1,\ldots,x_9\}$ which is a critically $3$-chromatic
graph.  We now apply the construction of Mycielski \cite{M} to make
a new graph $G'$ as follows:  let $V_{G'} = 
\{x_1,\ldots,x_9,y_1,\ldots,y_9,z\}$ where $z$ is adjacent to $y_1,\ldots,y_9$,
$y_i$ is adjacent to the neighbors of $x_i$ for each $i=1,\ldots,9$, and the induced
graph on $\{x_1,\ldots,x_9\}$ is the original graph $G= C_9$. The new graph $G'$
has the property that it is critically $4$-chromatic.  Furthermore, in this new graph
$G'$, the set $W = \{y_1,\ldots,y_9\}$ is a maximum independent set.

We claim that the graph $G'[W]$ has the property that $\chi(G'[W]) = \chi(G') = 4$.
To see this, color the vertices $x_1,\ldots,x_9$ in the following order:
red, blue, green, red, green, orange, red, orange, blue.  One can color
vertices $\{y_{11},y_{12},\ldots,y_{91},y_{92}\}$ using only the colors green, blue
and orange.  Finally, one colors the vertex $z$ red.

Using Macaulay 2 \cite{M2}, we calculated the irredundant irreducible decomposition
of the $J^4$, where $J = J(G')$ is cover ideal of $G'$. (Computing
this decomposition will take about an hour on a standard laptop.) One of the
irredundant irreducible ideals appearing in the decomposition is the ideal
\[ (x_1^3,x_2^4,x_3^3,x_4^4,x_5^3,x_6^4,x_7^3,x_8^4,x_9^4,y_1^3,y_2^4,y_3^3,y_4^4,y_5^3,y_6^4,y_7^3,y_8^4,y_9^4,z^4).\]
By Theorem \ref{criticalcorrespondence}, the induced graph on the vertex set
\[\{x_{11},x_{12},x_{21},x_{31},x_{32},x_{41},x_{51},x_{52},x_{61},x_{71},x_{72},x_{81},x_{91},\]
\[
y_{11},y_{12},y_{21},y_{31},y_{32},y_{41},y_{51},y_{52},y_{61},y_{71},y_{72},y_{81},y_{91},z_{11}\}\]
in the graph $(G')^2$ is a critically $5$-chromatic graph.  This induced graph
is isomorphic to the graph one obtains by expanding $G'$ at the vertex set $Z = \{x_1,x_3,x_5,x_7,
y_1,y_3,y_5,y_7\}$, which is a maximal independent set.  
So, although expanding $G'$ at one maximum independent set
does not result in graph whose chromatic number increases, at another maximal independent set we will have this property, i.e., $\chi(G'[Z])=\chi(G')+1$.
\end{example}

\begin{remark}\label{strongerconj}
The above example suggests that a stronger version of Conjecture \ref{conjecture}
may hold, namely, one can pick the set $W$ so that
$W$ is also a {\it maximal independent set}.
\end{remark}

We now give an algebraic proof that Conjecture \ref{conjecture} is true 
provided that the fractional chromatic number of $G$ is ``close enough'' to 
the chromatic number of $G$.  We begin by recalling some needed definitions
and results.

\begin{definition} A {\bf $b$-fold coloring} of a graph $G$ is an assignment
to each vertex a set of $b$ distinct colors such that adjacent vertices
receive disjoint sets of colors.  The minimal number of colors
needed to give $G$ a $b$-fold coloring is called the {\bf $b$-fold
chromatic number} of $G$ and is denoted $\chi_b(G)$. 
\end{definition}

When $b = 1$, then $\chi_b(G) = \chi(G)$.  The following result, which appeared
in \cite{FHVT2}, relates the value of $\chi_b(G)$ to an ideal membership
question.

\begin{theorem}\label{bfold}
Let $G$ be a finite simple graph on $V = \{x_1,\dots,x_n\}$ with
cover ideal $J$. Then \[\chi_b(G) = \min\{d ~|~ (x_1\cdots x_n)^{d-b} \in J^d \}.\]
\end{theorem} 

If follows directly from the definition 
that the $b$-fold chromatic number has the {\bf subadditivity property}, that is, $\chi_{a+b}(G) \le \chi_a(G) + \chi_b(G)$ for all $a$ and $b$. Thus, we can define the following invariant of $G$ (cf. \cite[Section 3.1 and Appendix 4]{SU}):

\begin{definition}  The {\bf fractional chromatic number} of a graph $G$, denoted $\chi_f(G)$,
is defined by
\[\chi_f(G) := \lim_{b \rightarrow \infty} \frac{\chi_b(G)}{b} = \inf_b \frac{\chi_b(G)}{b}.\]
\end{definition}

Note that $\chi_f(G) \le \dfrac{\chi_a(G)}{a}$ for all $a$. Moreover, by \cite[Corollary 1.3.2]{SU}, there always exists a number $b$ such that $\chi_f(G) = \dfrac{\chi_b(G)}{b}$.  This enables us to prove the following case of Conjecture \ref{conjecture}.

\begin{theorem} \label{specialcase}
Let $G$ be a finite simple graph that is
critically $s$-chromatic.  Suppose that $\chi(G) -1 < \chi_f(G) \leq \chi(G)$.
Then there exists a set $W \subseteq V_G$ such
that $\chi(G[W])$ is critically $(s+1)$-chromatic.
\end{theorem}

We will, in fact, prove a stronger result, by showing that we can
pick $W$ to be a maximal independent set, thus suggesting a stronger
version of Conjecture \ref{conjecture} holds (see Remark \ref{strongerconj}).
We need a technical algebraic result.

\begin{lemma}\label{technicallemma}
Let $W \subseteq V_G$ be any subset of vertices and set $G' = G[W]$.
Then
\[\frac{(x_1\cdots x_n)^{\chi_b(G') - b}}
{{\bf m}_{W}^{b}} \in J(G)^{\chi_b(G')} ~~\mbox{for any integer $b \geq 1$},\]
where ${\bf m}_{W} = \prod_{x_i \in W} x_i$.
\end{lemma}

\begin{proof}
For simplicity, let $W = \{x_1,\ldots,x_t\}$,
and we assume that 
\[V_{G'} = \{x_1,y_1,\ldots,x_t,y_t,x_{t+1},\ldots,x_n\},\] i.e.,
we use $x_i$ and $y_i$ to represent the shadows of $x_i$.  
By Theorem \ref{bfold}, we have
\[(x_1\cdots x_ny_1\cdots y_t)^{\chi_b(G')-b} \in J(G')^{\chi_b(G')}.\]
This means that 
\[(x_1\cdots x_ny_1\cdots y_t)^{\chi_b(G')-b} = m_1\cdots m_{\chi_b(G')}M
~\mbox{for some monomial $M$,}\]
where the $m_i$ are minimal generators of $J(G')$.
We can write each $m_i$ as $m_i = m_{i,x}m_{i,y}$ where $m_{i,x}$
is a monomial in $\{x_1,\ldots,x_n\}$ and $m_{i,y}$ is a monomial
in the $y_i$'s.  

We note that each $m_{i,x} \in J(G)$;  to see this, observe that the support of $m_i$
that is contained in $\{x_1,\ldots,x_n\}$ would cover the
vertices of $G$.  
Thus
\[(x_1\cdots x_n)^{\chi_b(G')-b} = m_{1,x}\cdots m_{\chi_b(G'),x}M_x \in 
J(G)^{\chi_b(G')},\]
where $M = M_xM_y$.
Note that if $y_i \nmid m_j$, then $x_i$ and $N(x_i)$, the neighbors of $x_i$, must divide
$m_j$, since the neighbors of $y_i$ are $x_i$ and $N(x_i)$.  In
other words, $x_i$ and $N(x_i)$ both divide $m_{j,x}$.  But then
$m_{j,x}$ is not a minimal vertex cover since we can remove $x_i$
and still have a vertex cover.  That is, $m_{j,x}/x_i \in J(G)$.

Now, for $i=1,\ldots,t$, the variable
$y_i$ is missing from at least $b$ of $\{m_1,\ldots,m_{\chi_b(G')}\}$ because
each of these monomials is square-free.
This means that for $i=1,\ldots,t$, the variable $x_i$ divides (at least) $b$ of the 
elements of $\{m_{1,x},\ldots,m_{\chi_b(G'),x}\}$,
and the resulting monomials still correspond to vertex covers.
Hence
\[\frac{(x_1\cdots x_n)^{\chi_b(G')-b}}{(x_1 \cdots x_t)^b} 
=\frac{m_{1,x}\cdots m_{\chi_b(G'),x}M_x}{(x_1 \cdots x_t)^b} \in J(G)^{\chi_b(G')}.\]
\end{proof}

\noindent{\bf Proof of Theorem \ref{specialcase}.}
By Lemma \ref{independentlemma}, it suffices to find a
maximal independent set $W$ such that $\chi(G[W]) = s+1$.  Suppose,
for a contradiction, that for every maximal
independent set $W$, $\chi(G[W]) = \chi(G)$.  

Fix any maximal independent set $W$,
and let  $G' = G[W]$.  There exists an integer $b$
such that $\chi_f(G') = \frac{\chi_b(G')}{b}$.  Furthermore, by Lemma \ref{technicallemma}
\[\frac{(x_1\cdots x_n)^{\chi_b(G')-b}}{{\bf m}_W^b} \in J(G)^{\chi_b(G')}.\]
(Note that the value of $b$ will depend upon the choice of $W$; the fact
remains that for every maximal independent set $W$, there exists
an integer $b$ so that the above monomial belongs to some power of
$J(G)$.)

Now we consider $\chi_f(G)$.  Again, there exists an integer
$c$ such that $\chi_f(G) = \frac{\chi_c(G)}{c}$.  We thus have
$(x_1\cdots x_n)^{\chi_c(G)-c} \in J(G)^{\chi_c(G)}$ by
Theorem \ref{bfold}.
This means that there exist minimal generators of $J(G)$ such that
\[(x_1\cdots x_n)^{\chi_c(G) -c} = m_{W_1}m_{W_2} \cdots m_{W_{\chi_c(G)}}M.\]
Since $m_{W_1}$ corresponds to a minimal vertex cover $W_1$, the 
complement $V_G \setminus W_1$ is a maximal
independent set of $G$.  Hence, by the discussion in the previous
paragraph, we also have 
\[(x_1 \cdots x_n)^{\chi_b(G') - 2b}m_{W_1}^b \in J(G)^{\chi_b(G')},\]
where $\frac{\chi_b(G')}{b} = \chi_f(G')$ and $G'$ is the expansion
of the independent set $V \setminus W_1$.

When we combine together the above results, we get:
\begin{eqnarray*}
(x_1\cdots x_n)^{\chi_b(G')-2b}m_{W_1}^b(m_{W_2}\cdots m_{W_{\chi_c(G)}}M)^b 
&\in& J(G)^{\chi_b(G')+b(\chi_c(G)-1)}\\
(x_1\cdots x_n)^{\chi_b(G')-2b + b\chi_c(G) -cb} &\in & 
J(G)^{\chi_b(G')+b\chi_c(G)-b}\\
(x_1\cdots x_n)^{(\chi_b(G') + b\chi_c(G) - b)-(c+1)b} &\in & 
J(G)^{\chi_b(G')+b\chi_c(G)-b}
\end{eqnarray*}
The last expression, coupled with Theorem \ref{bfold}, implies
that $(c+1)b$-fold chromatic number of $G$ is bounded above by
\[\chi_{(c+1)b}(G) \leq \chi_b(G') + b\chi_c(G) - b.\]
If we divide both sides of the equation by $(c+1)b$ we then get
\begin{eqnarray*}
\frac{\chi_{(c+1)b}(G)}{(c+1)b} &\leq& \frac{\chi_b(G')}{(c+1)b} + 
\frac{b\chi_c(G)}{(c+1)b} - \frac{b}{(c+1)b} 
=  \frac{\chi_f(G')+ \chi_c(G) - 1}{c+1},
\end{eqnarray*}
where we use the fact that $\chi_f(G') = \frac{\chi_b(G')}{b}$.

We have the following inequalities and equalities, where the first inequality
is our hypothesis, and the last equality comes from
our assumption that $\chi(G[W]) = \chi(G)$ for every maximal
independent set:
\[\chi(G)-1 < \chi_f(G) \leq \chi_f(G') \leq \chi(G') = \chi(G).\]
So, $\chi_f(G') < \chi_f(G)+1$, whence
\begin{eqnarray*}
\frac{\chi_f(G')+ \chi_c(G)-1}{c+1} & < & \frac{\chi_f(G)+1 + \chi_c(G) -1}{c+1}
 =  \frac{\chi_f(G)+ \chi_c(G)}{c+1} \\
&=& \frac{\chi_c(G)/c + \chi_c(G)}{c+1} =  \frac{\frac{\chi_c(G) + c\chi_c(G)}{c}}{c+1} \\
& = & \frac{(c+1)\frac{\chi_c(G)}{c}}{c+1} 
 =  \frac{\chi_c(G)}{c} = \chi_f(G).
\end{eqnarray*}
For any integer $a$, we have $\chi_f(G) \leq \frac{\chi_a(G)}{a}$.
So, if we put our inequalities together, we get 
\[\chi_f(G) \leq \frac{\chi_{(c+1)b}(G)}{(c+1)b} < \chi_f(G)\]
which gives our desired contradiction.  Hence there exists
a maximal independent set such that when we expand it, we get 
a graph whose chromatic number goes up.
\qed

\begin{corollary}
Conjecture \ref{conjecture} holds for the following critical graphs:
cliques, odd holes, and odd antiholes.
\end{corollary}

\begin{proof}  It suffices to show that for each type of graph, $\chi(G) -1 < \chi_f(G) 
\leq \chi(G)$.  For cliques, this is immediate since $\chi_f(G) = \chi(G)$.
When $G = C_{2n+1}$ is an odd hole,  $\chi_f(C_{2n+1}) = 2 + \frac{1}{n}$ and $\chi(G) = 3$,
so the result holds.  Finally, when $G = C_{2n+1}^c$, the chromatic number
of $G$ is $n+1$, while $\chi_f(G) = n+\frac{1}{2}$, so again the result holds.
\end{proof}

\begin{remark} Critically $3$-chromatic graphs are induced odd cycles. Let $C$ be an induced odd cycle on the vertices $\{x_1, \dots, x_{2r+1}\}$ for some $r \ge 1$. Let $W = \{x_2, x_4, \dots, x_{2r}\}$. Then it is not hard to see that $\chi(C[W]) = 4 = \chi(C) + 1$. This says that Conjecture \ref{conjecture} holds for $s = 3$. As a consequence, for any cover ideal $J = J(G)$ of a graph $G$, we always have 
$$\Ass(R/J^2) \subseteq \Ass(R/J^3).$$
In \cite{FHVT1}, we characterize elements of $\Ass(R/J^2)$. Thus, all primes of the form $(x_i,x_j)$, where $\{x_i,x_j\}$ is an edge in $G$, and $P = (x_{i_1}, \dots, x_{i_r})$, where $r$ is odd and the induced subgraph of $G$ on $\{x_{i_1}, \dots, x_{i_r}\}$ is an odd cycle, belong to $\Ass(R/J^3)$.
\end{remark}


\end{document}